\documentclass[11pt]{amsart}
\usepackage{}
\usepackage{amssymb}
\usepackage{amsfonts}
\usepackage{mathrsfs}
\usepackage{latexsym}
\usepackage{graphicx}
\usepackage{amscd,amssymb,amsmath,amsbsy,amsthm,amsfonts}
\usepackage[all]{xy}
\usepackage[colorlinks,plainpages,backref,urlcolor=blue]{hyperref}
\usepackage{verbatim}
\usepackage{enumerate}

\usepackage [english]{babel}
\usepackage [autostyle, english = american]{csquotes}
\MakeOuterQuote{"}

\newcommand {\Omit}[1]{}

\usepackage{tikz-cd}

\usepackage{tikz}
\usetikzlibrary{automata,positioning}

\usetikzlibrary{arrows,calc}
\tikzset{
>=stealth',
help lines/.style={dashed, thick},
axis/.style={<->},
important line/.style={thick},
connection/.style={thick, dotted},
}
\usetikzlibrary{patterns}
\newlength{\hatchspread}
\newlength{\hatchthickness}
\tikzset{hatchspread/.code={\setlength{\hatchspread}{#1}},
         hatchthickness/.code={\setlength{\hatchthickness}{#1}}}
\tikzset{hatchspread=3pt,
         hatchthickness=0.4pt}
\pgfdeclarepatternformonly[\hatchspread,\hatchthickness]
   {custom north west lines}
   {\pgfqpoint{-2\hatchthickness}{-2\hatchthickness}}
   {\pgfqpoint{\dimexpr\hatchspread+2\hatchthickness}{\dimexpr\hatchspread+2\hatchthickness}}
   {\pgfqpoint{\hatchspread}{\hatchspread}}
   {
    \pgfsetlinewidth{\hatchthickness}
    \pgfpathmoveto{\pgfqpoint{0pt}{\hatchspread}}
    \pgfpathlineto{\pgfqpoint{\dimexpr\hatchspread+0.15pt}{-0.15pt}}
    \pgfusepath{stroke}
   }

\newcommand{\nc}{\newcommand}
\nc{\rnc}{\renewcommand}
\nc{\bb}[1]{{\mathbb #1}}
\nc{\bbA}{\bb{A}}\nc{\bbB}{\bb{B}}\nc{\bbC}{\bb{C}}\nc{\bbD}{\bb{D}}
\nc{\bbE}{\bb{E}}\nc{\bbF}{\bb{F}}\nc{\bbG}{\bb{G}}\nc{\bbH}{\bb{H}}
\nc{\bbI}{\bb{I}}\nc{\bbJ}{\bb{J}}\nc{\bbK}{\bb{K}}\nc{\bbL}{\bb{L}}
\nc{\bbM}{\bb{M}}\nc{\bbN}{\bb{N}}\nc{\bbO}{\bb{O}}\nc{\bbP}{\bb{P}}
\nc{\bbQ}{\bb{Q}}\nc{\bbR}{\bb{R}}\nc{\bbS}{\bb{S}}\nc{\bbT}{\bb{T}}
\nc{\bbU}{\bb{U}}\nc{\bbV}{\bb{V}}\nc{\bbW}{\bb{W}}\nc{\bbX}{\bb{X}}
\nc{\bbY}{\bb{Y}}\nc{\bbZ}{\bb{Z}}
\nc{\mbf}[1]{{\mathbf #1}}
\nc{\bfA}{\mbf{A}}\nc{\bfB}{\mbf{B}}\nc{\bfC}{\mbf{C}}\nc{\bfD}{\mbf{D}}
\nc{\bfE}{\mbf{E}}\nc{\bfF}{\mbf{F}}\nc{\bfG}{\mbf{G}}\nc{\bfH}{\mbf{H}}
\nc{\bfI}{\mbf{I}}\nc{\bfJ}{\mbf{J}}\nc{\bfK}{\mbf{K}}\nc{\bfL}{\mbf{L}}
\nc{\bfM}{\mbf{M}}\nc{\bfN}{\mbf{N}}\nc{\bfO}{\mbf{O}}\nc{\bfP}{\mbf{P}}
\nc{\bfQ}{\mbf{Q}}\nc{\bfR}{\mbf{R}}\nc{\bfS}{\mbf{S}}\nc{\bfT}{\mbf{T}}
\nc{\bfU}{\mbf{U}}\nc{\bfV}{\mbf{V}}\nc{\bfW}{\mbf{W}}\nc{\bfX}{\mbf{X}}
\nc{\bfY}{\mbf{Y}}\nc{\bfZ}{\mbf{Z}}
\nc{\bfa}{\mbf{a}}\nc{\bfb}{\mbf{b}}\nc{\bfc}{\mbf{c}}\nc{\bfd}{\mbf{d}}
\nc{\bfe}{\mbf{e}}\nc{\bff}{\mbf{f}}\nc{\bfg}{\mbf{g}}\nc{\bfh}{\mbf{h}}
\nc{\bfi}{\mbf{i}}\nc{\bfj}{\mbf{j}}\nc{\bfk}{\mbf{k}}\nc{\bfl}{\mbf{l}}
\nc{\bfm}{\mbf{m}}\nc{\bfn}{\mbf{n}}\nc{\bfo}{\mbf{o}}\nc{\bfp}{\mbf{p}}
\nc{\bfq}{\mbf{q}}\nc{\bfr}{\mbf{r}}\nc{\bfs}{\mbf{s}}\nc{\bft}{\mbf{t}}
\nc{\bfu}{\mbf{u}}\nc{\bfv}{\mbf{v}}\nc{\bfw}{\mbf{w}}\nc{\bfx}{\mbf{x}}
\nc{\bfy}{\mbf{y}}\nc{\bfz}{\mbf{z}}

\nc{\mcal}[1]{{\mathcal #1}}
\nc{\calA}{\mcal{A}}\nc{\calB}{\mcal{B}}\nc{\calC}{\mcal{C}}\nc{\calD}{\mcal{D}}
\nc{\calE}{\mcal{E}} \nc{\calF}{\mcal{F}}\nc{\calG}{\mcal{G}}\nc{\calH}{\mcal{H}}
\nc{\calI}{\mcal{I}}\nc{\calJ}{\mcal{J}}\nc{\calK}{\mcal{K}}\nc{\calL}{\mcal{L}}
\nc{\calM}{\mcal{M}}\nc{\calN}{\mcal{N}}\nc{\calO}{\mcal{O}}\nc{\calP}{\mcal{P}}
\nc{\calQ}{\mcal{Q}}\nc{\calR}{\mcal{R}}\nc{\calS}{\mcal{S}}\nc{\calT}{\mcal{T}}
\nc{\calU}{\mcal{U}}\nc{\calV}{\mcal{V}}\nc{\calW}{\mcal{W}}\nc{\calX}{\mcal{X}}
\nc{\calY}{\mcal{Y}}\nc{\calZ}{\mcal{Z}}
\nc{\fA}{\frak{A}}\nc{\fB}{\frak{B}}\nc{\fC}{\frak{C}} \nc{\fD}{\frak{D}}
\nc{\fE}{\frak{E}}\nc{\fF}{\frak{F}}\nc{\fG}{\frak{G}}\nc{\fH}{\frak{H}}
\nc{\fI}{\frak{I}}\nc{\fJ}{\frak{J}}\nc{\fK}{\frak{K}}\nc{\fL}{\frak{L}}
\nc{\fM}{\frak{M}}\nc{\fN}{\frak{N}}\nc{\fO}{\frak{O}}\nc{\fP}{\frak{P}}
\nc{\fQ}{\frak{Q}}\nc{\fR}{\frak{R}}\nc{\fS}{\frak{S}}
\nc{\fU}{\frak{U}}\nc{\fV}{\frak{V}}\nc{\fW}{\frak{W}}\nc{\fX}{\frak{X}}
\nc{\fY}{\frak{Y}}\nc{\fZ}{\frak{Z}}
\nc{\fa}{\frak{a}}\nc{\fb}{\frak{b}}\nc{\fc}{\frak{c}} \nc{\fd}{\frak{d}}
\nc{\fe}{\frak{e}}\nc{\fFf}{\frak{f}}\nc{\fg}{\frak{g}}\nc{\fh}{\frak{h}}
\nc{\fri}{\frak{i}}\nc{\fj}{\frak{j}}\nc{\fk}{\frak{k}}\nc{\fl}{\frak{l}}
\nc{\fm}{\frak{m}}\nc{\fn}{\frak{n}}\nc{\fo}{\frak{o}}\nc{\fp}{\frak{p}}
\nc{\fq}{\frak{q}}\nc{\fr}{\frak{r}}\nc{\fs}{\frak{s}}
\nc{\fu}{\frak{u}}\nc{\fv}{\frak{v}}\nc{\fw}{\frak{w}}\nc{\fx}{\frak{x}}
\nc{\fy}{\frak{y}}\nc{\fz}{\frak{z}}

\newtheorem{theorem}{Theorem}[section]

\newtheorem{prop}[theorem]{Proposition}

\theoremstyle{definition}
\newtheorem{definition}[theorem]{Definition}

\newtheorem{thm}{Theorem}

\newtheorem{cor}{Corollary}

 \DeclareMathOperator{\gr}{gr}

 \DeclareMathOperator{\Sym}{Sym}

  \DeclareMathOperator{\KM}{KM}
   
      \DeclareMathOperator{\sph}{sph}

\DeclareMathOperator{\cross}{cross}

\DeclareMathOperator{\inc}{in}
\DeclareMathOperator{\out}{out}

\DeclareMathOperator{\Res}{Res}

\DeclareMathOperator{\Sh}{Sh}

\newcommand{\surj}{\twoheadrightarrow}

\newcommand{\C}{\bbC}

\newcommand{\N}{\bbN}

\DeclareMathOperator{\fin}{fin}

\DeclareMathOperator{\fac}{fac}

\DeclareMathOperator{\ext}{ext}
\DeclareMathOperator{\coop}{coop}
\DeclareMathOperator{\loc}{loc}

\DeclareMathOperator{\SH}{SH}

\newcount\cols
{\catcode`,=\active\catcode`|=\active
 \gdef\Young(#1){\hbox{$\vcenter
 {\mathcode`,="8000\mathcode`|="8000
  \def,{\global\advance\cols by 1 &}%
  \def|{\cr
        \multispan{\the\cols}\hrulefill\cr
        &\global\cols=2 }%
  \offinterlineskip\everycr{}\tabskip=0pt
  \dimen0=\ht\strutbox \advance\dimen0 by \dp\strutbox
  \halign
   {\vrule height \ht\strutbox depth \dp\strutbox##
    &&\hbox to \dimen0{\hss$##$\hss}\vrule\cr
    \noalign{\hrule}&\global\cols=2 #1\crcr
    \multispan{\the\cols}\hrulefill\cr%
   }
 }$}}
}

\newcommand{\yaping}[1]{\textcolor{blue}{$[$ Yaping: #1 $]$}}

\begin{document}
\title{The PBW theorem for the affine Yangians}
\date{\today}

\author[Y.~Yang]{Yaping~Yang}
\address{School of Mathematics and Statistics, The University of Melbourne, 813 Swanston Street, Parkville VIC 3010, Australia}
\email{yaping.yang1@unimelb.edu.au}

\author[G.~Zhao]{Gufang~Zhao}
\address{Institute of Science and Technology Austria,
Am Campus, 1,
Klosterneuburg 3400,
Austria}
\email{gufang.zhao@ist.ac.at}


\begin{abstract}
We prove that the Yangian associated to an untwisted symmetric affine Kac-Moody Lie algebra is isomorphic to the Drinfeld double of a shuffle algebra. The latter is constructed in \cite{YZ1} as an algebraic formalism of the cohomological Hall algebras. As a consequence, we obtain the Poincare-Birkhoff-Witt (PBW) theorem for this class of affine Yangians. Another independent proof of the PBW theorem is given recently by Guay-Regelskis-Wendlandt \cite{GW}. 
\end{abstract}
\thanks{We are grateful to Hiraku Nakajima for the helpful suggestions, and to
 Nicolas Guay and Curtis Wendlandt for the useful discussions and communications. 
The research of G.Z. at IST Austria, Hausel group, is supported by
the Advanced Grant Arithmetic and Physics of Higgs moduli spaces No. 320593 of the European Research Council.
}

\maketitle

\section*{Introduction}
For Yangians of a semi-simple Lie algebra, the PBW theorem is well-known. It has been conjectural that 
the PBW theorem holds for a much larger class of Yangians including the affine Yangians.

On the other hand, the PBW theorem for the (undeformed) current algebra of an affine Lie algebra has been proved by Enriquez \cite{E}. In the proof of {\it loc. cit.}, a shuffle algebra description and a duality of the current algebra play the key roles. 

In the framework of cohomological Hall algebras, in \cite{YZ1} the authors of the present paper gave a conjectural shuffle algebra description of the Yangian of any symmetric Kac-Moody Lie algebra, and constructed a surjective  map from the Yangian to the shuffle algebra. In the present paper, we combine this result and the earlier results of Enriquez \cite{E} to prove the PBW theorem  for the Yangian of any untwisted symmetric affine Lie algebra.

For $Y_\hbar(\widehat{\mathfrak{sl}}_{n})$, this theorem has been proved by N. Guay in \cite{G}. 
During the preparation of the present paper, Guay-Regelskis-Wendlandt communicated to the authors an independent proof of the PBW theorem for $Y_\hbar(\fg_{\KM})$   using the vertex operator representations of the affine Yangians \cite{GW}.

This paper is organized as follows. 
We start by recollecting relevant results from \cite{YZ2}. 
The statements of the main theorems are in \S~\ref{subsec:mainThm}.
In \S~\ref{subsec:maingoal} we recollect some results from \cite{E} and state the PBW theorem for the positive part of the affine Yangian. 
The proof of this statement is in \S~\ref{sec:positive}. Finally, in \S~\ref{sec:triang} we give a triangular decomposition of the entire affine Yangian and hence prove the PBW theorem.

\section{The affine Yangian and the shuffle algebra}

\subsection{The Yangian}
Let $A=(c_{ij})_{i, j\in I}$ be a symmetrizable generalized Cartan matrix. Thus, $c_{ii}=2$ for any $i\in I$, $c_{ij}\in \mathbb{Z}_{\leq 0}$ for any $i\neq j\in I$. Let $\fg_{\KM}$ be the Kac-Moody Lie algebra associated to $A$. Denote by $\fh$ the Cartan subalgebra of $\fg_{\KM}$, and $\{\alpha_i\}_{i\in I}$ a set of simple roots of $\fg_{\KM}$. 
The algebra $Y_\hbar(\fg_{\KM})$ is denoted by $Y_\hbar(\fg'_{\KM})$ in \cite{GNW} which is the version of the Yangian of $\fg_{\KM}$ without the degree operator. 
\begin{definition}
The Yangian $Y_\hbar(\fg_{\KM})$, is an associative algebra over $\C[\hbar]$, generated by the variables
\[
x_{k, r}^{\pm}, \xi_{k, r}, (k\in I, r\in \N),
\]
subject to relations described below. 
Take the generating series $\xi_{k}(u), x_k^{\pm}(u)\in Y_{\hbar}(\fg_Q)[\![u^{-1}]\!]$ by
\[
\xi_{k}(u)=1+\hbar \sum_{r\geq 0} \xi_{k, r} u^{-r-1}\,\ \text{and}\,\ 
x_k^{\pm}(u)=\hbar \sum_{r\geq 0} x_{k, r}^{\pm} u^{-r-1}. \]
The following is a complete set of relations defining $Y_\hbar(\fg_{\KM})$ (see, e.g.,  \cite[Proposition 2.3]{GTL16}):

\makeatletter
\let\orgdescriptionlabel\descriptionlabel
\renewcommand*{\descriptionlabel}[1]{%
  \let\orglabel\label
  \let\label\@gobble
  \phantomsection
  \edef\@currentlabel{#1}%
  \let\label\orglabel
  \orgdescriptionlabel{#1}%
}
\makeatother

\begin{description}
\item[(Y1)\label{Y1}]
For any $i, j\in I$, and $h, h'\in \fh$
\begin{align*} 
[\xi_{i}(u), \xi_{i}(v)]=0, \,\  [\xi_{i}(u), h]=0,\,\  [h, h']=0
\end{align*}
\item[(Y2)\label{Y2}] 
For any $i\in I$, and $h \in \fh$,
 \[[h, x_{i}^{\pm}(u)]=\pm \alpha_{i}(h) x_{i}^{\pm}(u)\]
 \item[(Y3) \label{Y3}]
 For any $i, j\in I$, and $a=\frac{\hbar c_{ij}}{2}$
 \[
 (u-v\mp a) \xi_i(u)x_j^{\pm}(v)
=(u-v\pm a) x_j^{\pm}(v) \xi_i(v) \mp 2a x_{j}^{\pm}(u\mp a)\xi_{i}(u)
 \]
 \item[(Y4)\label{Y4}] 
 For any $i, j\in I$, and $a=\frac{\hbar c_{ij}}{2}$
 \[
 (u-v\mp a) x_i^{\pm}(u)x_j^{\pm}(v)
=(u-v\pm a) x_j^{\pm}(v) x_i^{\pm}(u)+ \hbar
\Big( [x_{i,0}^{\pm}, x_j^{\pm}(v)]-[x_i^{\pm}(u), x_{j, 0}^{\pm}]\Big) \]
 \item[(Y5) \label{Y5}] 
 For any $i, j\in I$, 
 \[
 (u-v)[x^+_i (u),x^-_j (v)]=-\delta_{ij} \hbar(\xi_i(u)-\xi_i(v))
 \]
 \item[(Y6) \label{Y6}]
 For any $i\neq j\in I$, 
 \[
 \sum_{\sigma\in\fS_{1-c_{ij}}}[x_i^{\pm}(u_{\sigma(1)}),[x_i^{\pm}(u_{\sigma(2)}),[\cdots,[x_i^{\pm}(u_{\sigma (1-c_{ij})}),x_j^{\pm}(v)]\cdots]]]=0. \]
\end{description}
\end{definition}

Let $Y_\hbar^+(\fg_{\KM})$ (resp. $Y_\hbar^-(\fg_{\KM})$) be the algebra generated by the elements $x_{k, r}^{+}$, (resp. $x_{k, r}^{-}$) for $k\in I, r\in \N$, subject to the relations \ref{Y4} and \ref{Y6} with the $+$ sign (resp. with the $-$ sign). 
Let $Y_\hbar^{\geq 0}(\fg_{\KM})$ be the algebra generated by the elements $x_{k, r}^{+}$, $h_{k, r}$, for $k \in I, r\in \N$, subject to the relations \ref{Y1}, \ref{Y2}, \ref{Y3}, \ref{Y4}, and \ref{Y6} with the $+$ sign. 

Let $Y^0=Y^0_\hbar(\fg_{\KM})$ be the polynomial algebra generated by $\xi_{k, r}, (k\in I, r\in \N)$. 
We note that from the relation \ref{Y3} it is clear that there is a surjective map $Y_\hbar^+(\fg_{\KM})\rtimes Y^0\to Y_\hbar^{\geq 0}(\fg_{\KM})$ (see, e.g., \cite[Lemma~2.9]{GTL}).
Each of these algebras has an obvious algebra homomorphism to $Y_\hbar(\fg_{\KM})$.
 However, it is {\it a priori} not clear that $Y_\hbar^+(\fg_{\KM})$ or $ Y_\hbar^{\geq 0}(\fg_{\KM})$ are subalgebras of $Y_\hbar(\fg_{\KM})$. These facts only follow from \S~\ref{sec:triang}.

\subsection{The shuffle algebra}
The shuffle algebra associated to any formal group law is introduced in \cite{YZ1, YZ2}. It is an algebraic description of the cohomological Hall algebra (COHA) associated to a preprojective algebra. The shuffle formula in \cite{YZ3} describes the action of the COHA on the Nakajima quiver varieties algebraically. In this section, we focus on the case when the formal group law is additive, and recall the definition of the shuffle algebra introduced in \cite{YZ1}. 

\subsubsection{The algebra structure}\label{subsec:shuff}
Now we assume the Cartan matrix of $\fg_{\KM}$ is symmetric.
Let $Q=(I, H)$ be the associated Dynkin quiver of the symmetric Kac-Moody Lie algebra $\fg_{\KM}$ with vertex set $I$ and arrow set $H$. For each arrow $h\in H$, we denote by $\inc(h)$ (resp. $\out(h)$) the incoming (resp. outgoing) vertex of $h$.
Thus, the vertex set $I$ is the index set of the simple roots of $\fg_{\KM}$, and the arrow set $H$ encodes the Cartan matrix of $\fg_{\KM}$.  

Let $\SH$ be an $\bbN^I$-graded $\mathbb{C}[\hbar]$-algebra. As a $\mathbb{C}[\hbar]$-module, we have 
\[\SH=\bigoplus_{\vec{v}\in\bbN^I}\SH_{\vec{v}}, \,\ \text{where}\,\ \SH_{\vec{v}}:=\mathbb{C}[\hbar]\otimes \mathbb{C}
[ \lambda^i_s ]_{i\in I, s=1,\dots, v^i}^{\fS_{\vec{v}}},\]
here $\fS_{\vec{v}}=\prod_{i\in I} \fS_{v^i}$ is the product of symmetric groups, and $\fS_{\vec{v}}$ naturally acts on the variables $\{ \lambda^i_s \}_{i\in I, s=1,\dots, v^i}$ by permutation. 
For any $\vec{v_1}$ and $\vec{v_2}\in \bbN^I$, 
we consider $\SH_{\vec{v_1}}\otimes_{\mathbb{C}[\hbar]} \SH_{\vec{v_2}}$ as a subalgebra of $
\mathbb{C}[\hbar][\lambda^i_s]_{\{ i\in I, s=1,\dots, ({v_1^i}+{v_2^i}) \}}$ 
by sending $\lambda'^i_s  \in \SH_{\vec{v_1}} $ to $\lambda^i_s$, and 
$\lambda''^i_s \in \SH_{\vec{v_2}} $ to $\lambda^i_{s+v_1^i}$.

The factor $\fac_{\vec{v_1}, \vec{v_2}}$ is defined as follows. 
For any pair of vertices $i$ and $j$ with arrows  $h_1, \dots, h_a$ from $i$ to $j$, we associated to each arrow $h_p$ the pairs of integers  $m_{h_p}=a+2-2p$ and $m_{h_p^*}=-a+2p$. Then,
\begin{align*}
\fac_{\vec{v_1}, \vec{v_2}}&:=
\prod_{i\in I}\prod_{s=1}^{v_1^i}
\prod_{t=1}^{v_2^i}\frac{\lambda\rq{}^i_s- \lambda\rq{}\rq{}^i_t+ \hbar}
{\lambda\rq{}^i_s-\lambda\rq{}\rq{}^i_t }\cdot \\
&\cdot \prod_{h\in H}\Big(
\prod_{s=1}^{v_1^{\out(h)}}
\prod_{t=1}^{v_2^{\inc(h)}}
(\lambda_t^{'' \inc(h)}- \lambda_s^{'\out(h)}+  m_h \frac{\hbar}{2})
\prod_{s=1}^{v_1^{\inc(h)}}
\prod_{t=1}^{v_2^{\out(h)}}
(\lambda_s^{'\inc(h)}-\lambda_t^{''\out(h)}-m_{h^*} \frac{\hbar}{2})
\Big). 
\end{align*}

We define the shuffle product $\SH_{\vec{v_1}}\otimes_{\mathbb{C}[\hbar]} \SH_{\vec{v_2}}\to \SH_{\vec{v_1}+\vec{v_2}}$ by
\begin{align}
&f(\lambda_{\vec{v_1}} )\otimes g(\lambda_{\vec{v_2}})\mapsto
\sum_{\sigma\in \Sh(\vec{v_1}, \vec{v_2})}  \sigma\Big(f(\lambda'_{\vec{v_1}})\cdot g(\lambda''_{\vec{v_2}})\cdot \fac_{\vec{v_1},\vec{v_2}}\Big), \label{shuffle formula}
\end{align}
where 
$\Sh(\vec{v_1}, \vec{v_2}):=\prod_{i\in I} \Sh(v_1^i, v_2^i)$, and $\Sh(v_1^i, v_2^i)$ consists of $(v_1^i, v_2^i)$-shuffles, i.e., permutations of $\{1,\cdots ,v_1^i+v_2^i\}$ that preserve the relative order of 
$\{1,\cdots ,v_1^i\}$ and $\{v_1^i+1, \cdots, v_1^i+v_2^i\}$.

This $\fac_{\vec{v_1}, \vec{v_2}}$ is slightly different than the one in \cite{YZ1, YZ2} in that
here  the sign twist in \cite{YZ1} is absorbed into the $\fac_{\vec{v_1}, \vec{v_2}}$ used here.

A special case of the following fact is discussed in \cite[Example 3.6]{YZ3}. 
Due to technical reason we also need to consider another version of the shuffle algebra $\SH'$, whose underlying vector space is a localization of that of $\SH$, denoted by 
$\SH'=\bigoplus_{\vec{v}\in \N^I} (\SH'_{\vec{v}})_{\loc}$ to differentiate it from $\SH$. The multiplication 
$f_1(\lambda_{\vec{v_1}})\star' f_2(\lambda_{\vec{v_2}})$ is defined to be
\begin{align}
\sum_{\{\sigma\in \hbox{Sh}(v_1, v_2)\}}
\sigma\cdot \Big(f_1 \cdot f_2 \cdot \prod_{i\in I}\prod_{s=1}^{v_1^i}
\prod_{t=1}^{v_2^i}\frac{\lambda\rq{}^i_s- \lambda\rq{}\rq{}^i_t+ \hbar}
{\lambda\rq{}^i_s-\lambda\rq{}\rq{}^i_t }
\prod_{i, j\in I}
\prod_{s=1}^{v_2^{i}}
\prod_{t=1}^{v_1^{j}}
\frac{\lambda_t^{'j}-\lambda_s^{''i}-a_{ij}\frac{\hbar}{2}
}{\lambda_t^{'j}- \lambda_s^{''i}+ a_{ij}\frac{\hbar}{2}} \Big).
\label{eq:fac'}
\end{align}

\begin{prop}\label{prop:fac'}
There is an algebra homomorphism
\begin{align*}
\bigoplus_{\vec{v}\in \N^I} \SH_{\vec{v}} \to \bigoplus_{\vec{v}\in \N^I} (\SH'_{\vec{v}})_{\loc}, \,\ \text{given by} \,\ f(\lambda_{\vec{v}})\mapsto \frac{ f(\lambda_{\vec{v}})}{H_{\vec{v}}(\lambda_{\vec{v}})}, 
\end{align*}
where $
H_{\vec{v}}(\lambda_{\vec{v}}):=
\prod_{h\in H}
\prod_{s=1}^{v^{\out(h)}}
\prod_{t=1}^{v^{\inc(h)}}
(\lambda_t^{\inc(h)}- \lambda_s^{\out(h)}+  m_h \frac{\hbar}{2})$. 
In particular, when restricting on $\SH_{\vec{e_k}}$, the above homomorphism is an identity, where $\vec{e_k}$ is the dimension vector valued $1$ at vertex $k$ and zero otherwise.
\end{prop}
Here the localization is taken in the obvious sense. As a free polynomial ring is an integral domain, this map described above is injective. 
\begin{proof}
This map is well-defined, since the factor $H_{\vec{v}}(\lambda_{\vec{v}})$ is invariant under the action of $\hbox{Sh}(\vec{v_{1}}, \vec{v_{2}})$.

Define $
H_{\cross, \vec{v_1}, \vec{v_2}}:
=\frac{H_{\vec{v_1}+\vec{v_2}}(\lambda_{\vec{v_1}}\cup \lambda_{\vec{v_2}})}{
H_{\vec{v_1}}(\lambda_{\vec{v_1}})H_{\vec{v_2}}(\lambda_{\vec{v_2}})
}, $ we have
\begin{align*}
H_{\cross, \vec{v_1}, \vec{v_2}}
\Omit{=&\prod_{h\in H}
\frac{
\prod_{s=1}^{(v_1+v_2)^{\out(h)}}
\prod_{t=1}^{(v_1+v_2)^{\inc(h)}}
(\lambda_t^{\inc(h)}- \lambda_s^{\out(h)}+  m_h \frac{\hbar}{2})}
{
\prod_{s=1}^{v_1^{\out(h)}}
\prod_{t=1}^{v_1^{\inc(h)}}
(\lambda_t^{\inc(h)}- \lambda_s^{\out(h)}+  m_h \frac{\hbar}{2})
\prod_{s=1+v_1^{\out(h)}}^{v_1^{\out(h)}+v_2^{\out(h)}}
\prod_{t=1+v_1^{\inc(h)}}^{v_1^{\inc(h)}+v_2^{\inc(h)}}
(\lambda_t^{\inc(h)}- \lambda_s^{\out(h)}+  m_h \frac{\hbar}{2})}\\}
=&\prod_{h\in H}
\prod_{s=1}^{v_2^{\out(h)}}
\prod_{t=1}^{v_1^{\inc(h)}}
(\lambda_t^{'\inc(h)}- \lambda_s^{''\out(h)}+  m_h \frac{\hbar}{2})
\prod_{s=1}^{v_1^{\out(h)}}
\prod_{t=1}^{v_2^{\inc(h)}}
(\lambda_t^{''\inc(h)}- \lambda_s^{'\out(h)}+  m_h \frac{\hbar}{2})
\end{align*}
Dividing $\fac_{\vec{v_1}, \vec{v_2}}$ by $H_{\cross, \vec{v_1}, \vec{v_2}}$, we obtain
\begin{align*}
\frac{\fac_{\vec{v_1}, \vec{v_2}}}{H_{\cross, \vec{v_1}, \vec{v_2}} }
&=
\prod_{i\in I}\prod_{s=1}^{v_1^i}
\prod_{t=1}^{v_2^i}\frac{\lambda\rq{}^i_s- \lambda\rq{}\rq{}^i_t+ \hbar}
{\lambda\rq{}^i_s-\lambda\rq{}\rq{}^i_t }
\prod_{h\in H}
\prod_{s=1}^{v_2^{\out(h)}}
\prod_{t=1}^{v_1^{\inc(h)}}
\frac{\lambda_t^{'\inc(h)}-\lambda_s^{''\out(h)}-m_{h^*} \frac{\hbar}{2}
}{\lambda_t^{'\inc(h)}- \lambda_s^{''\out(h)}+  m_h \frac{\hbar}{2}}\\
&=
\prod_{i\in I}\prod_{s=1}^{v_1^i}
\prod_{t=1}^{v_2^i}\frac{\lambda\rq{}^i_s- \lambda\rq{}\rq{}^i_t+ \hbar}
{\lambda\rq{}^i_s-\lambda\rq{}\rq{}^i_t }
\prod_{i, j\in I}
\prod_{p=1}^{a_{ij}}
\prod_{s=1}^{v_2^{i}}
\prod_{t=1}^{v_1^{j}}
\frac{\lambda_t^{'j}-\lambda_s^{''i}+(a_{ij}-2p) \frac{\hbar}{2}
}{\lambda_t^{'j}- \lambda_s^{''i}+  (a_{ij}+2-2p) \frac{\hbar}{2}}\\
&=
\prod_{i\in I}\prod_{s=1}^{v_1^i}
\prod_{t=1}^{v_2^i}\frac{\lambda\rq{}^i_s- \lambda\rq{}\rq{}^i_t+ \hbar}
{\lambda\rq{}^i_s-\lambda\rq{}\rq{}^i_t }
\prod_{i, j\in I}
\prod_{s=1}^{v_2^{i}}
\prod_{t=1}^{v_1^{j}}
\frac{\lambda_t^{'j}-\lambda_s^{''i}-a_{ij}\frac{\hbar}{2}
}{\lambda_t^{'j}- \lambda_s^{''i}+ a_{ij}\frac{\hbar}{2}}. 
\end{align*}
It implies 
\[
\frac{f_{1}(\lambda_{\vec{v_1}}) \star f_{2}(\lambda_{\vec{v_2}})}{H_{\vec{v}}(\lambda_{\vec{v}})}
=\frac{f_{1} (\lambda_{\vec{v_1}})}{ H_{\vec{v_1}}(\lambda_{\vec{v_1}})} \star'
 \frac{f_{2}(\lambda_{\vec{v_2}})}{ H_{\vec{v_2}}(\lambda_{\vec{v_2}})}. 
\]
This completes the proof. 
\end{proof}

For each $k\in I$, recall that $\vec{e_k}$ is the dimension vector valued $1$ at vertex $k$ and zero otherwise. That is, $\vec{e_k}$ corresponds to the simple root $\alpha_k$ of $\fg_{\KM}$. 

\begin{definition}
We define the spherical subalgebra $\SH^{\sph}$ to be the subalgebra of $\SH$ generated by $\SH_{\vec{e_k}}$ as $k$ varies in $I$.\end{definition}
In \cite[\S4.2]{YZ2}, we constructed the reduced Drinfeld double $\overline{D}(\SH^{\sph,\ext})$ of $\SH$. The construction is given as follows. 
Roughly speaking, let $\SH^0:=\Sym(\fh[u])$ be the symmetric algebra of $\fh[u]$, which acts on $\SH^{\sph}$ by \cite[Lemma 1.4]{YZ2}. Let $\SH^{\ext, \sph}:=\SH^0\ltimes \SH^{\sph}$ be the extended shuffle algebra. $\SH^{\ext, \sph}$ is a bialgebra endowed with a non-degenerate bialgebra pairing, whose coproduct is given in \cite[Theorem 2.2]{YZ2} and the pairing in \cite[Theorem 3.5]{YZ2}. Let $D(\SH^{\sph,\ext})$ be the Drinfeld double of $\SH^{\sph,\ext}$, as a vector space one has 
$D(\SH^{\sph,\ext})\cong \SH^{\sph,\ext}\otimes (\SH^{\sph,\ext})^{\coop}$.
Define the reduced Drinfeld double $\overline{D}(\SH^{\sph,\ext})$ to be the quotient of $D(\SH^{\sph,\ext})$ by identifying $\SH^0$ in 
$\SH^{\sph,\ext}$ with $(\SH^0)^{\coop}$ in $(\SH^{\sph,\ext})^{\coop}$. See \cite[\S 3.2]{YZ2} for the details.  In particular, we have the triangular decomposition 
\[
\overline{D}(\SH^{\sph,\ext})\cong  \SH^{\sph}\otimes \SH^0\otimes \SH^{\sph,\coop}. 
\]

\subsubsection{The pairing $(, )_{\SH}$ on $\SH^{\sph, \ext}$}
\label{sec:pair2}
This part follows from the results in \cite{YZ2}. We briefly recall the non-degenerate pairing on the extended shuffle algebra $\SH^{\sph,\ext}=\SH^{\sph}\rtimes \SH^0$ in \cite[Section 3]{YZ2}. 

 Let $\SH_\bbA^{\sph,\ext}:=\SH^{\sph}_\bbA\rtimes \SH^0_\bbA$ be the ad\'ele version of the shuffle algebra defined in detail in \cite[\S~3.1]{YZ1}. The non-degenerate bialgebra pairing on $\SH_\bbA^{\sph,\ext}$ is defined as follows. 
\[
( \cdot , \cdot) : \SH_\bbA^{\sph,\ext} \otimes (\SH_\bbA^{\sph,\ext})^{\coop} \to \bbC
\]
\begin{itemize}
\item
For $f\in \SH_{\bbA, \vec{v}}^{\sph,\ext}$, and $g\in \SH_{\bbA, \vec{w}}^{\sph,\ext}$, we define $(f, g)=0$ if $\vec{v}\neq \vec{w}$. 
\item
For $h\in \SH^0$, and $f\in \SH_{\bbA}^{\sph}$, we define $(h, f)=0$. 
\item
For $H_k(u) \in \SH^0[\![u]\!]$, we define 
$(H_k(u), H_k(w))=\frac{\fac(u|w)}{\fac(w|u)}$, for any $k\in I$, where $H_k(u)=1+\hbar \sum_{r\geq 0} h_{k, r} u^{-r-1}$ is the generating series
of the standard basis $\{h_{k,r}=h_k\otimes \lambda^{r} \mid k\in I, r\in\bbN\}$ of $\SH^0=\Sym(\fh[\lambda])$. 
Here $\fac(u|w)$ is $\fac_{\vec{e_k},\vec{e_k}}$ with the variables $\lambda^{k'}$ and $\lambda^{k''}$ replaced by $u$ and $w$ respectively. 
\item
For any $i\in I$ and $f, g\in \SH_{\bbA, \vec{e_i}}$, we follow Drinfeld \cite{D86} and define 
\[
(f, g):=\Res_{x=\infty} (f_x\cdot g_{-x} dx). 
\]
\end{itemize}
In general, for $f, g\in \SH_{\bbA,\vec{v}}^{\sph,\ext}$, the pairing $(f, g)$ is defined to be 
\[
(f, g):= \sum_{x\in \mathbb{P}^1}\Res_{x } \Big( \frac{f(x_{\vec{v}})\cdot g(-x_{\vec{v}})}{ {\prod_{i\in I} (v^i)!} \fac(x_{\vec{v}})} dx \Big), 
\]
where
\begin{align*}
\fac(x_{\vec{v}})&:=\prod_{i\in I} \prod_{\{s, t \mid 1\leq s\neq t \leq v^i \}}
\frac{\lambda\rq{}^i_s- \lambda\rq{}\rq{}^i_t+ \hbar}{\lambda\rq{}\rq{}^i_t- \lambda\rq{}^i_s}\cdot \\
&\cdot \prod_{h\in H}\Big(
\prod_{s=1}^{v^{\out(h)}}
\prod_{t=1}^{v^{\inc(h)}}
(\lambda_t^{'' \inc(h)}- \lambda_s^{'\out(h)}+  m_h \frac{\hbar}{2})
\prod_{s=1}^{v^{\inc(h)}}
\prod_{t=1}^{v^{\out(h)}}
(\lambda_t^{''\out(h)}- \lambda_s^{'\inc(h)}+ m_{h^*} \frac{\hbar}{2})
\Big). 
\end{align*}
 
This pairing on $\SH^{\sph}_{\bbA}\otimes \SH^{\sph,\coop}_{\bbA}$ described above is a non-degenerate bialgebra pairing, which restricts to a non-degenerate one  when ${\hbar=0}$. In particular, on the set of generators, we have $\fac(x_{\vec{e_i}})=1$, and 
\begin{align*}
((\lambda^{(i)})^{r}, (\lambda^{(j)})^{s})
=&\delta_{ij} \Res_{x=\infty} (x^r \cdot (-x)^s dx)\\
 =&(-1)^{s}\delta_{ij}\delta_{r+s, -1}.
\end{align*}

\begin{theorem}\cite[Theorem B]{YZ2}
\label{thmfromYZ2}
Let $Y_\hbar(\fg_{\KM})$ be the Yangian endowed with the Drinfeld comultiplication.
Assume $Q$ has no edge-loops. Then there is a bialgebra epimorphism  
\[
\Psi: Y_\hbar(\fg_{\KM}) \surj \overline{D}(\SH^{\sph,\ext}).
\]
When $Q$ is of finite type, this map is an isomorphism. 
\end{theorem}
Composing $\Psi$ with the natural map $Y^+_\hbar(\fg_{\KM})\to Y_\hbar(\fg_{\KM})$, we get the
 map $\Psi: Y^+_\hbar(\fg_{\KM}) \surj \SH^{\sph}$, which  sends $x_{k, r}^+$ to $(\lambda^{(k)})^{r}\in \SH_{\vec{e_k}}= \C[\hbar][\lambda^{(k)}]$, where $k\in I$. 
Finally we recall that from \cite[\S~4.2]{YZ2}
composing $\Psi$ with the natural map $Y^-_\hbar(\fg_{\KM})\to Y_\hbar(\fg_{\KM})$ we get a surjective algebra homomorphism 
\begin{equation}\label{eqn:coop}
Y^{\leq 0}_\hbar(\fg)\to (\SH^{\sph,\ext})^{\coop}
\end{equation} 
defined by sending $x^-_{i,r}$ to $(-\lambda^{(i)})^{r}$  and  $\xi_{k,s}$ to $-(-h^{(k)})^s$.

\subsection{The main theorems}
\label{subsec:mainThm}
From now on till the end of this paper, we take $\fg_{\KM}$ to be of {\bf untwisted symmetric affine type}. The simple roots $I$ are labeled by $0,1,\cdots, n$ with $0$ being the extended node in the affine Dynkin diagram. 

In the present paper, we prove the following. 
\begin{thm}\label{thm:A}
Assume the symmetric Cartan matrix $A$ is of nontwisted affine Kac-Moody type and not of type $A_{1}^{(1)}$. The morphism $\Psi: Y_\hbar(\fg_{\KM}) \to \overline{D}(\SH^{\sph,\ext})$ is an isomorphism. 
\end{thm}
Theorem \ref{thm:A} is proved in a  similar  fashion to \cite{E} and \cite{Neg}. In \cite{E}, the PBW theorem for the quantum Kac-Moody algebras is proved using an isomorphism with the Feigin-Odesskii shuffle algebras. In \cite{Neg}, an isomorphism of the quantum toroidal algebra of $\mathfrak{gl}_n$ and the double shuffle algebra of Feigin-Odesskii is established. 

As a consequence of Theorem \ref{thm:A}, we have the following PBW theorem of the affine Yangian $Y_\hbar(\fg_{\KM})$. 

Define a filtration on $Y_\hbar(\fg_{\KM})$ by $\deg x^{\pm}_{i,r}\leq r$ and $\deg h_{i,r}\leq r$ for all $i\in I$ and $r\in\bbN$.
Let $\gr(Y_\hbar(\fg_{\KM}))$ be the associated graded algebra. Let $\mathfrak{t}$ be the universal central extension of $\fg_{\KM}[u]$. 
\begin{cor}\label{cor1}
We have the natural isomorphism $\gr(Y_\hbar(\fg_{\KM})) \cong  U(\mathfrak{t})[\hbar]$. 
\end{cor}

In \S~\ref{sec:triang}, we also prove a triangular decomposition of $Y_\hbar(\fg_{\KM})$.

\section{The positive part}
\label{subsec:maingoal}
In this section, we focus on the positive part  $Y^+_\hbar(\fg_{\KM})$ of the Yangian, and study the map $\Psi: Y^+_{\hbar}(\fg_{\KM})\surj \SH^{\sph}$. 

\subsection{Preliminaries on the current algebras} 

We collect some facts from \cite[\S1.3]{E}. Let $\fg_{\fin}$ be the finite dimensional Lie algebra of $\fg_{\KM}$. 
Let $\mathfrak{g}_{\fin}=\fn_{\fin}^-\oplus \fh_{\fin}\oplus \fn_{\fin}$ be the triangular decomposition of $\mathfrak{g}_{\fin}$. Then, $\fn_{\KM}=(\lambda \bbC[\lambda]\otimes (\fn_{\fin}^-+\fh_{\fin})+\bbC[\lambda]\otimes  \fn_{\fin})$ \cite[\S 7.6]{K}.

\subsubsection{A presentation}
The set up in \cite[\S1.3]{E} is for the loop algebra $\fn_{\KM}[u, u^{-1}]$. For our purposes, we state the results for the current algebra $\fn_{\KM}[u]$. 

Define $\mathfrak{t}^+$ to be the direct sum $\fn_{\KM}[u]\oplus\bigoplus_{k>0, l \in \mathbb{N}} \bbC K_{k\delta}[l]$, and endow it with the bracket such that the elements $K_{k\delta}[l]$ are central, and 
\[
[(x\otimes u^l, 0), (y\otimes u^{m}, 0)]
=([x\otimes u^l, y\otimes u^{m}], \langle \bar{x}, \bar{y} \rangle(lk''-mk'')K_{(k'+k'')\delta}[l+m]), 
\]
where $x\mapsto \bar{x}\otimes \lambda^{k'}, y\mapsto \lambda^{k''}$ by the inclusion $\fn_{\KM}\subset \mathfrak{g}_{\fin}[\lambda, \lambda^{-1}]$, and 
$\langle\bar{x}, \bar{y}\rangle$ is an invariant bilinear form on $\mathfrak{g}_{\fin}$. 

\begin{definition}\label{def:F+}
Let $\tilde{F}^{+}$ be the Lie algebra with generators $x_{i, k}^+$, $i\in I, k\in\bbN$, and relations given by \ref{Y4} with $\hbar=0$, and \ref{Y6}. In other words, $U(\tilde{F}^{+})=Y^+_{\hbar=0}(\fg_{\KM})$.  
\end{definition}
It is straightforward to see that there is a Lie algebra morphism 
\[
j^+: \tilde{F}^{+}\to \fn_{\KM}[u], \,\ x_{i, k}^+\mapsto x_{i}\otimes u^k.
\] 

 \begin{prop}\cite[Proposition~1.6]{E} \label{prop:fromE}
 \begin{enumerate}
 \item When $A$ is of affine Kac-Moody type, the kernel $j^+$ is equal to the center of $\tilde{F}^{+}$, so that $\tilde{F}^{+}$ is a central extension of $\fn_{\KM}[u]$. 
 \item We have a unique Lie algebra map $j': \mathfrak{t}^{+} \to \tilde{F}^{+}$ such that $j'(e_i\otimes t^k)=x_{i, k}^+$, where $i=0, 1, \cdots, n$ This map is an isomorphism iff $A$ is not of type $A_{1}^{(1)}$. If $A$ is of type $A_{1}^{(1)}$, $j'$ is surjective, and its kernel is $\oplus_{k\in \bbN}\bbC K_{\delta}[k]$. 
 \end{enumerate}
  \end{prop}
As a consequence, when $A$ is of affine Kac-Moody type and not of type $A_{1}^{(1)}$, there is an algebra homomorphism 
$U(\mathfrak{t}^+) \to U( \fn_{\KM}[u])$ extending the identity map on $\fn_{\KM}[u]$. The kernel of this map  is generated by $\bigoplus_{k>0, l \in \mathbb{N}} K_{k\delta}[l]$.

\subsubsection{The pairing $(, )_{\mathfrak{t}}$ on $U(\mathfrak{t}^+)$}
\label{subsec:3.1}
This part follows from the results in \cite{E}. 

We first recall the pairing defined on a toroidal algebra in \cite[Remark~4.3]{E}. Let $\mathfrak{tor}$ be the universal central extension of $\fg_{\fin}[\lambda^{\pm}, u^{\pm}]$. Then,  
\[
\mathfrak{tor}=\fg_{\fin}[\lambda^{\pm}, u^{\pm}]\oplus Z, 
\] where $
Z=\Omega^1_{ \bbC[\lambda^{\pm}, u^{\pm}]}/d( \bbC[\lambda^{\pm}, u^{\pm}])
=\bigoplus_{k, l \in \mathbb{Z}} K_{k\delta}[l] \oplus\bbC c$, and $K_{k\delta}[l]$ is the class
of $\frac{1}{k} \lambda^k  u^{l-1} du$ if $k\neq 0$, $K_{0}[l]$ is the class $u^{l} \frac{d\lambda}{\lambda}$, 
$c$ is the class $\frac{du}{u}$. We decompose $Z$ as $Z=Z_{>}\oplus Z_{0}\oplus Z_{<}$, where
 \[
 \text{$Z_{>}:=\bigoplus_{k>0, l \in \mathbb{Z}} K_{k\delta}[l]$, \,\
 $Z_{0}:=\bigoplus_{l \in \mathbb{Z}} K_{0}[l]\oplus \bbC c$ \,\ and \,\ 
$Z_{<}:=\bigoplus_{k<0, l \in \mathbb{Z}} K_{k\delta}[l]$.}
\] 

Enlarge $\mathfrak{tor}$ by $\tilde{\mathfrak{tor}}:=
\fg_{\fin}[u^{\pm}, \lambda^{\pm}]\oplus Z \oplus D$, 
where 
\[D=D_{>}\oplus D_{0}\oplus D_{<}= 
(\bigoplus_{k>0, l \in \mathbb{Z}} D_{k\delta}[l])  \oplus (\bigoplus_{l \in \mathbb{Z}} D_{0}[l]\oplus \bbC d)\oplus (\bigoplus_{k<0, l \in \mathbb{Z}} D_{k\delta}[l]).
\]
The following pairing on $\tilde{\mathfrak{tor}}$ can be found in \cite[Remark~4.3]{E}. 
\begin{prop} 
\label{prop:pairing}
There is a non-degenerate bialgebra pairing $(, )_{\tilde{\mathfrak{tor}}}$ on $\tilde{\mathfrak{tor}}=\fg_{\fin}[\lambda^{\pm}, u^{\pm}]\oplus Z \oplus D$. 
Let $x, x'\in \fg_{\fin}$, the pairing $(, )_{\tilde{\mathfrak{tor}}}$ is the following. 
\begin{align*}
&(x\otimes \lambda^au^b, x'\otimes \lambda^{a'} u^{b'})_{\tilde{\mathfrak{tor}}}
=\langle x, x'\rangle_{\fg_{\fin}} \delta_{a+a', 0} \delta_{b-b', 0}, \\
& (K_{k\delta}[l], D_{k'\delta}[l'])_{\tilde{\mathfrak{tor}}}=\delta_{k+k', 0} \delta_{l-l', 0}, \,\, (d, c)_{\tilde{\mathfrak{tor}}}=1. 
\end{align*}
and all other pairings of elements 
$x\otimes \lambda^a u^b, K_{k\delta}[l], D_{k\delta}[l], c, d$ are zero.
\end{prop}
Note that we follow the same convention as in \cite[\S~3.1]{YZ2} on parings on a Manin triple, which differs from the convention in \cite{E} by taking the group inverse on the second variable. 
The Lie algebra $\mathfrak{t}^+$ is isomorphic to the Lie subalgebra of $ \tilde{\mathfrak{tor}}$ \cite[pg. 29]{E}
\[
\fn_{\KM}[u]\oplus\bigoplus_{k>0, l \in \mathbb{N}} \bbC K_{k\delta}[l]=
(\lambda \bbC[\lambda]\otimes (\fn_{\fin}^-+\fh_{\fin})+\bbC[\lambda]\otimes  \fn_{\fin})[u]\oplus\bigoplus_{k>0, l \in \mathbb{N}} \bbC K_{k\delta}[l].
\] 
Let $\mathfrak{t}^-\subset \tilde{\mathfrak{tor}}$ be the subalgebra
$\fn_{\KM}^{-}[u]\oplus\bigoplus_{k<0, l\in \bbN} \bbC D_{k\delta}[l]=
(\lambda^{-1} \bbC[\lambda^{-1}]\otimes (\fn_{\fin}^++\fh_{\fin})+\bbC[\lambda^{-1}]\otimes  \fn_{\fin})[u]\oplus\bigoplus_{k<0, l\in \bbN} \bbC D_{k\delta}[l].$

Restricting the pairing in Proposition~\ref{prop:pairing},  we get a non-degenerate pairing 
$\mathfrak{t}^+\times \mathfrak{t}^-\to \mathbb{C}$. We now describe this paring on the 
 set of generators. We have, for $i, j=0, 1, \cdots, n$, 
\begin{equation}\label{eqn:toroidal}
(x_{i, r}^+, x_{j, s}^-)_{\mathfrak{t}}=\delta_{ij}\delta_{r-s, 0}, \,\ (K_{k\delta}[l], D_{k'\delta}[l'])=\delta_{k+k', 0} \delta_{l-l', 0}. 
\end{equation}
Indeed, 
when $i=j=0$, $x_{0, r}^+$ corresponds to $E_{-\theta}\otimes \lambda u^r$, and 
$x_{0, s}^-$ corresponds to $E_{\theta}\otimes \lambda^{-1} u^s$, where $\theta$ is the highest root of $\mathfrak{g}_{\fin}$, and $E_{\theta}\in \mathfrak{g}_{\fin,\theta}$
\cite[\S 7.4]{K}. Therefore,  $(x_{0, r}^+, x_{0, s}^-)_{\mathfrak{t}}=\delta_{ij}\delta_{r-s, 0}$.

This pairing identifies $\mathfrak{t}^+$ with the dual of $\mathfrak{t}^-$ endowed with the opposite coproduct.  Note that however, $\mathfrak{t}^+$ is isomorphic to $\mathfrak{t}^-$ as algebras via the following map  
\begin{equation} \label{eqn:t+}
E_{i,-k-1}\mapsto F_{i,k} 
\end{equation}
for all $i=0,\dots,n$ and $k\in\bbZ$.
Using the presentation of $\mathfrak{t}^+$ given in Proposition~\ref{prop:fromE}, one easily verify that the relations hold in $\mathfrak{t}^-$ by comparing with \cite[\S2.3, QL4]{GTL}. 
Moreover, $H_{i,r}\mapsto -H_{i,-r}$ extends the above isomorphism to the Cartan subalgebra.

\subsection{The statement} 
Motivated by the filtration on $Y^+_\hbar(\fg_{\KM})$, we define a filtration $F=\{F_{r}\}_{r\in \N}$ on $\SH^{\sph}$, such that $(\lambda^{(k)})^{r}$ is in the filtered piece $F_r$.
As various filtrations on $\SH^{\sph}$ will be considered, we denote this filtration by $F=\{F_{r}\}_{r\in \N}$ and the corresponding associated graded algebra by $\gr_F(\SH^{\sph})$. That is, the generator $(\lambda^{(k)})^{r}$ is in the filtered piece $F_r$, and $\hbar\in F_0$. 
Recall that the algebra homomorphism in  Theorem \ref{thmfromYZ2} is a filtered map, hence gives rise to a surjective morphism on the associated graded algebras
\[
\gr(\Psi): \gr(Y^+_{\hbar}(\fg_{\KM})) \surj \gr_F(\SH^{\sph}).
\] 
Using the isomorphism $U(\mathfrak{t}^+) \cong U(\tilde{F}^{+})$ and the presentation of the latter in Definition \ref{def:F+}, we have an algebra epimorphism $\pi: U(\mathfrak{t}^+)[\hbar] \cong U(\tilde{F}^{+})[\hbar] \surj \gr(Y^+_{\hbar}(\fg_{\KM})), e_i\otimes t^k\mapsto x_{i, k}^+$, $i=0, 1, \cdots, n$.

\begin{theorem}\label{thm:main}
The composition 
\[
\xymatrix{
U(\mathfrak{t}^+)[\hbar] \ar@{->>}[r]^(0.4){\pi} & \gr(Y^+_{\hbar}(\fg_{\KM})) \ar@{->>}[r]^{\gr(\Psi)}& \gr_F(\SH^{\sph})
}
\]
 is an isomorphism. 
\end{theorem}
We will prove Theorem \ref{thm:main} in \S~\ref{sec:positive}. 
By Theorem \ref{thmfromYZ2}, $\Psi$ is an epimorphism.
It follows from this theorem that  $\gr(\Psi)$ is an isomorphism, and consequently that 
$\Psi: Y^+_{\hbar}(\fg_{\KM})\cong \SH^{\sph}$ is an isomorphism.

\section{Proof of Theorem \ref{thm:main}}
\label{sec:positive}
We prove Theorem~\ref{thm:main} in the following steps.
\subsection{Step 1}
Note that $Y^+_{\hbar}(\fg_{\KM})|_{\hbar=0}\cong U(\mathfrak{t}^+)$ by Proposition \ref{prop:fromE}. 
In this step, we show the following. 
\begin{prop}\label{prop: hbar=0}
The map $\Psi|_{\hbar=0}: U(\mathfrak{t}^+)\surj \SH^{\sph}|_{\hbar=0}$ is an algebra isomorphism.  
\end{prop}

The map  $\Psi|_{\hbar=0}$ is compatible with the coproducts:
\begin{equation}\label{eq2}
\Psi|_{\hbar=0}( \Delta(x))=\Delta(\Psi|_{\hbar=0}(x)), \,\ \text{for $x\in \mathfrak{t}^+$}.
\end{equation}
In \cite[Section 2]{YZ2}, we defined a coproduct 
$\Delta: \SH^{\ext, \sph}\to \SH^{\ext, \sph}\hat{\otimes} \SH^{\ext, \sph}$ on the extended shuffle algebra. 
In particular, by \cite[(8)]{YZ2}, 
\[
\Delta( (\lambda^{(i)})^{r})=H_{i}(\lambda^{(i)})\otimes (\lambda^{(i)})^{r}+(\lambda^{(i)})^{r}\otimes 1.
\] 
When specializing $\hbar=0$, we have on $\SH^{\sph}_{\hbar=0}$, 
$\Delta( (\lambda^{(i)})^{r})=1\otimes (\lambda^{(i)})^{r}+(\lambda^{(i)})^{r}\otimes 1$.
As the coproduct on $U(\mathfrak{t}^+)$ is given by 
$\Delta( x_{i, r}^+)=1\otimes  x_{i, r}^++ x_{i, r}^+\otimes 1$, 
 we have, on the level of generators,
\[
\Psi|_{\hbar=0}( \Delta(x_{i, r}^+))
=\Delta( (\lambda^{(i)})^{r})
=\Delta(\Psi|_{\hbar=0}(x_{i, r}^+)).\]
The claim \eqref{eq2} now follows from the fact that $\Delta$ and $\Psi|_{\hbar=0}$ are algebra homomorphisms. 

We prove Proposition~\ref{prop: hbar=0} using the non-degenerate bialgebra pairings on $U(\mathfrak{t}^+)$ and $\SH^{\sph}|_{\hbar=0}$ described below. The non-degenerate bialgebra pairings on $U(\mathfrak{t}^+)$ and $ \SH^{\sph}|_{\hbar=0}$ are denoted by $(, )_{\mathfrak{t}}$ and $(, )_{\SH_{\hbar=0}}$  respectively.
The paring $(, )_{\SH_{\hbar=0}}$ is recollected in \S~\ref{sec:pair2}. 
We now explain the pairing $(, )_{\mathfrak{t}}$, and the compatibility of the map $\Psi|_{\hbar=0}$ with the parings.

We have the following commutative diagram of algebras
\[\xymatrix@C=4em{
U(\mathfrak{t}^+)^{\coop}\ar[r]^{\Psi^{\coop}|_{\hbar=0}}&\SH_{\hbar=0}^{\sph,\coop}\\
U(\mathfrak{t}^-)^{\coop}\ar[u]_{\cong}&Y^-_{\hbar=0}(\fg)\ar[l]_{\cong} \ar[u]
}\]
where the left vertical map is from \eqref{eqn:toroidal}; the right vertical map is from \cite[\S~4.2]{YZ1} recalled in \eqref{eqn:coop} which in particular is surjective. The bottom isomorphism is similar to Proposition~\ref{prop:fromE}. The map $\Psi^{\coop}$ sends $x_{j, s}^+$ to $(-1)^s(\lambda^{(j)})^{s}$. The commutativity of this diagram is clear by keeping track of the generators of $Y^-_\hbar(\fg)$. 

Identifying $\mathfrak{t}^-$ with $\mathfrak{t}^+$ using the map \eqref{eqn:t+} $x_{i,k}^- \leftrightarrow x^+_{i,-k-1}$, 
we get the pairing $(, )_{\mathfrak{t}}$ on $\mathfrak{t}^+$ using the paring described in  \S~\ref{subsec:3.1}. 
On the level of generators, we have
\[
(x_{i, r}^+, x_{j, s}^+)_{\mathfrak{t}}=
(x_{i, r}^+, x_{j, -1-s}^-)_{\mathfrak{t}^+\times \mathfrak{t}^-}=
\delta_{ij}\delta_{r+s+1, 0}. 
\]
Therefore, 
\begin{equation}\label{eq1}
(x_{i, r}^+, x_{j, s}^+)_{\mathfrak{t}}=
(\Psi|_{\hbar=0}(x_{i, r}^+), \Psi^{\coop}|_{\hbar=0}(x_{j, s}^+))_{\SH_{\hbar=0}}=
((\lambda^{(i)})^{r}, (-1)^s(\lambda^{(j)})^{s})_{\SH_{\hbar=0}}=\delta_{ij}\delta_{r+s, -1}. 
\end{equation}
Moreover, they are bialgebra pairings, i.e., 
\begin{equation}\label{eq3}
(a\star b, c)=(a\otimes b, \Delta(c)), \,\ \text{ for $a,b,c \in A$,}
\end{equation}
where $A=\mathfrak{t}^+$ or $\SH_{\hbar=0}$. 
Therefore, 
it follows from  \eqref{eq2}, \eqref{eq1}, and \eqref{eq3}  that the map $\Psi|_{\hbar=0}$ is compatible with the pairings on the entire $U(\mathfrak{t}^+)$ and $\SH^{\sph}|_{\hbar=0}$.

The injectivity of $\Psi|_{\hbar=0}$ follows from the nondegeneracy of the two pairings when extended to the ad\'ele version and the fact that the map $\Psi|_{\hbar=0}: U(\mathfrak{t}^+) \surj  \SH^{\sph}|_{\hbar=0}$ preserves the pairing. 
Indeed, for any $x\in U(\mathfrak{t}^+)$ such that $\Psi|_{\hbar=0}(x)=0$. We have, for any $y\in U(\mathfrak{t}^+)$,
\[
(x, y)_{\mathfrak{t}}=(\Psi|_{\hbar=0}(x), \Psi^{\coop}|_{\hbar=0}(y))_{\SH_{\hbar=0}}=0.
\] It implies $x=0$ by non-degeneracy. Therefore, $\Psi|_{\hbar=0}$ is injective. This concludes Proposition \ref{prop: hbar=0}. 

\subsection{Step 2}

\Omit{
\yaping{from Wiki}
A filtered algebra is an algebra $(A,\cdot )$ over $k$ which has an increasing sequence 
$\{\{0\}\subseteq F_{0}\subseteq F_{1}\subseteq \cdots \subseteq F_{i}\subseteq \cdots \subseteq A\}$
of subspaces of $A$ such that $ A=\bigcup _{i\in \mathbb {N} }F_{i}$, and that is compatible with the multiplication in the following sense:
$ \forall m,n\in \mathbb {N}, F_{m}\cdot F_{n}\subseteq F_{n+m}.$
}

We define another filtration on $\SH^{\sph}$ by the degrees of $\hbar$. To be more precisely, we assign 
$\deg(\hbar)=-1$, $\deg((\lambda^{(i)})^{r})=0$, and extend it to a filtration on $\SH^{\sph}$. Let $\gr_{\hbar}(\SH)$ be the associated graded of $\SH$ using this $\hbar$-filtration. 
\Omit{
\yaping{There is a problem with the filtration if I set $\deg(\hbar)=1$. In $\gr_{\hbar}(Y)$, I am looking at the top degree term of $\hbar$, and throwing away the lower degree of $\hbar$. If I look at the Yangian relation, I would have in  $\gr_{\hbar}(Y)$
\[
h_{ir} x_{js}+x_{js}h_{ir} =0, \text{and $x_{ir}x_{js}+x_{js}x_{ir}=0$, if $c_{ij}\neq 0$. }
\]
If the above is true, then using $x_{ir}x_{js}-x_{js}x_{ir}=[x_i, x_j]\otimes t^{r+s}$, we conclude 
$2(x_{i}\otimes t^{r})(x_{j}\otimes t^s)=[x_i, x_j]\otimes t^{r+s}$, which can't be true. 
So I change the degree $\hbar=-1$. }}

In this subsection, we show the following. 
\begin{prop}\label{prop2}
We have an isomorphism of algebras
\[
\gr_{\hbar}(\SH) \cong \SH|_{\hbar=0}[\hbar].
\]
\end{prop}
\begin{proof}
Let $F'=\{F'_r\}_{r\in \bbZ_{\leq 0}}$ be the the filtration induced by $\hbar$, defined as $\deg((\lambda^{(k)})^{r})=0$, and $\deg(\hbar)=-1$. Then, $F'_r$ consists of those elements in $\SH$ whose $\hbar^n$ term satisfies $n\geq -r$. 
Therefore, $F'_r/F'_{r-1}$ consists of those elements in $\SH$ whose $\hbar^n$ term satisfies $n=-r$. 
As a consequence, we have the isomorphism $F'_r/F'_{r-1}\cong (\SH|_{\hbar=0})\hbar^{-r}$. 
We note that this isomorphism, which relies on the fact that $\SH$ is flat as a $\bbC[\hbar]$-module, is the key to the proof of the PBW theorem in the present paper. 

We check that this isomorphism preserves the algebra structure. For any $x, y\in \gr_{\hbar}(\SH)$, without loss of generality, we can assume the leading terms of $x, y$ have degree $0$. 
Otherwise, we can shift $x, y$ by $\hbar^n$, for some $n\in \N$. 
The multiplication $\overline{x\star y}$ in $\gr_{\hbar}(\SH)$ is by throwing away the $\hbar$-terms in $x\star y\in \SH$. 
This is the same as setting $\hbar=0$. Therefore, it is the same as the multiplication in $\SH|_{\hbar=0}[\hbar]$. 
This finishes the proof.
 \end{proof}

\subsection{Step 3}
We now show the following statement. 
\begin{prop}\label{prop3}
We have an isomorphism of algebras
\[
\gr_F(\SH^{\sph})\cong \gr_{\hbar}(\SH^{\sph}). 
\]
\end{prop}
Note that this isomorphism of algebras does not preserve the gradings. 
\begin{proof}
Now we consider 
 a third grading on $\SH^{\sph}$ defined as follows. 
Let $\SH'$ be the localized shuffle algebra considered in \S~\ref{subsec:shuff}. It has a grading defined  by total polynomial degrees in all the variables $\hbar$ and $\lambda_t^{(k)}$. In particular,  $\deg((\lambda^{(k)})^{r})=r$, and $\deg(\hbar)=1$. 
Note that the multiplication formula  \eqref{eq:fac'} makes $\SH'$  into a graded algebra with the degree-$m$ piece denoted by $G_m'$ for $m\in\bbZ$. 
Using the embedding from Proposition~\ref{prop:fac'}, we define the grading on $\SH^{\sph}$ via $G_m=G_m'\cap \SH^{\sph}$. We therefore have $\SH^{\sph}=\bigoplus_{m\in \N} G_m$.

Recall the filtration $F$ on $\SH^{\sph}$ defined by $\deg((\lambda^{(k)})^{r})\leq r$, and $\deg(\hbar)=0$. It induces a filtration on $G_m$ for each $m\in\bbZ$
\[
(F_m\cap G_m) \supseteq (F_{m-1}\cap G_m) \supseteq \cdots \supseteq (F_{0}\cap G_m). \]
Similarly, let $F'=\{F'_r\}_{r\in \bbZ_{\leq 0}}$ be the the filtration induced by $\hbar$, defined as $\deg((\lambda^{(k)})^{r})=0$, and $\deg(\hbar)=-1$.
The filtration $F'$ on $\SH^{\sph}$ induces a filtration on $G_m$:
\[
(F'_0\cap G_m) \supseteq (F'_{-1}\cap G_m) \supseteq \cdots \supseteq (F'_{-m}\cap G_m). \]
Both $F$ and $F'$ are compatible with $G$ in the sense that $F_m=\oplus_i(G_i\cap F_m)$ and $F'_m=\oplus_i(G_i\cap F'_m)$.

We have the following observation, for $0\leq r\leq m$, 
\[
F_{r} \cap G_m =F'_{r-m} \cap G_m. 
\]
Indeed, we have 
\begin{align*}
x\in  F_{r} \cap G_m &\Longleftrightarrow \text{$x$ is of the form $x=\sum_{\{s\mid s\leq r\}} x_s \hbar^{m-s}$, where $\deg(x_s)=s$.}\\
 &\Longleftrightarrow x\in F'_{r-m} \cap G_m, 
\end{align*}
where the last implication follows from the fact that the power of $\hbar$ in $x$ showing up above are all  of the form $m-s \geq m-r$. 
This implies the isomorphism. 
\[
\bigoplus_{r=0}^m (F_{r} \cap G_m)/(F_{r-1} \cap G_m) \cong
\bigoplus_{r=-m}^0 (F'_{r} \cap G_m)/(F_{r+1}' \cap G_m). 
\]
Therefore, we have
\begin{align*}
\gr_F(\SH^{\sph}) 
&=  \bigoplus_{m\geq 0} \bigoplus_{r=0}^m (F_{r} \cap G_m)/(F_{r-1} \cap G_m) \\
&\cong \bigoplus_{m\geq 0}  \bigoplus_{r=-m}^0 (F'_{r} \cap G_m)/(F_{r-1}' \cap G_m)\\
&= \gr_{\hbar}(\SH^{\sph}). 
\end{align*}
The algebra structures on $\gr_F(\SH^{\sph})$ and $ \gr_{\hbar}(\SH^{\sph})$ are both induced from the product on $\SH^{\sph}$. 
Therefore, we conclude the claim. 
\end{proof}
\subsection{Conclusion}
We put Propositions \ref{prop: hbar=0}, \ref{prop2}, \ref{prop3} together. 
The morphism $\gr(\Psi) \circ \pi$ in Theorem \ref{thm:main} is the following composition
\[
U(\mathfrak{t}^+)[\hbar]\cong \SH^{\sph}|_{\hbar=0}[\hbar] \cong  \gr_{\hbar}(\SH^{\sph}) \cong \gr_F(\SH^{\sph}). 
\]
Therefore, $\gr(\Psi) \circ \pi$ is an isomorphism. This completes the proof of Theorem \ref{thm:main}. 

\section{Triangular decomposition}\label{sec:triang}
Recall we have the natural map $\Psi: Y_\hbar(\fg_{\KM})\to \overline{D}(\SH^{\sph,\ext})\cong \SH^{\sph}\otimes\SH^0\otimes\SH^{\sph,\coop}$ from Theorem~\ref{thmfromYZ2}, where the last isomorphism is only of vector spaces. Composing with the natural map $Y^+_\hbar(\fg_{\KM})\to Y_\hbar(\fg_{\KM})$, we get the map 
$Y^+_\hbar(\fg_{\KM})\to \SH^{\sph}\otimes\SH^0\otimes\SH^{\sph,\coop}$, which is injective thanks to Theorem~\ref{thm:main}.
When restricting on $Y^0$, the map $\Psi: Y^0_\hbar(\fg_{\KM}) \to \SH^0$ is an isomorphism by the definitions of both sides. 

In this section, we prove the following theorem.
\begin{theorem}
The map $\Psi: Y_\hbar(\fg_{\KM})\to \SH^{\sph}\otimes\SH^0\otimes\SH^{\sph,\coop}$ is an isomorphism of vector spaces. In particular, the natural map $Y^+_\hbar(\fg_{\KM})\otimes Y^0\otimes Y^-_\hbar(\fg_{\KM})\to  Y_\hbar(\fg_{\KM})$ induced by multiplications is an isomorphism of vector spaces. 
\end{theorem}
This is the triangular decomposition of  the  Yangian.
\begin{proof}
It follows from the commutation relation of $Y^+_\hbar(\fg_{\KM})$ and $ Y^0$ (see, e.g., \cite[Lemma~2.9]{GTL}) that the natural map
$Y^+_\hbar(\fg_{\KM})\otimes Y^0\to Y^{\geq0}_\hbar(\fg_{\KM})$ induced by multiplication is surjective. 
By symmetry, $Y^-_\hbar(\fg_{\KM})\otimes Y^0\to Y^{\leq0}_\hbar(\fg_{\KM})$ is also surjective.
Hence, the commutation relation of $Y^+_\hbar(\fg_{\KM})$ and $Y^-_\hbar(\fg_{\KM})$ \ref{Y5} implies that the natural map $Y^{+}_\hbar(\fg_{\KM})\otimes Y^0\otimes Y^-_\hbar(\fg_{\KM}) \to Y_\hbar(\fg_{\KM})$ is surjective.

Compositing this surjective map with the algebra epimorphism $\Psi:Y_\hbar(\fg_{\KM})\to \overline{D}(\SH^{\sph,\ext})$ from  Theorem~\ref{thmfromYZ2}, we get 
\[
Y^{+}_\hbar(\fg_{\KM})\otimes Y^0\otimes Y^-_\hbar(\fg_{\KM}) \to Y_\hbar(\fg_{\KM})\to \overline{D}(\SH^{\sph,\ext})\cong \SH^{\sph}\otimes\SH^0\otimes\SH^{\sph,\coop}.\]
Theorem~\ref{thm:main} yields that the composition is an 
 isomorphism of vector spaces. 
 
Therefore, $Y^{+}_\hbar(\fg_{\KM})\otimes Y^0\otimes Y^-_\hbar(\fg_{\KM}) \cong Y_\hbar(\fg_{\KM})$ and  $Y_\hbar(\fg_{\KM})\cong \overline{D}(\SH^{\sph,\ext})$. 
 This completes the proof. 

\end{proof}

\newcommand{\arxiv}[1]
{\texttt{\href{http://arxiv.org/abs/#1}{arXiv:#1}}}
\newcommand{\doi}[1]
{\texttt{\href{http://dx.doi.org/#1}{doi:#1}}}
\renewcommand{\MR}[1]
{\href{http://www.ams.org/mathscinet-getitem?mr=#1}{MR#1}}


\begin{thebibliography}{00}

\bibitem[D86]{D86}V. G. Drinfeld, {\it Quantum groups}, Proceedings of the International Congress of Mathematicians 1986, Vol. 1, 798--820, AMS 1987.

\bibitem[E03]{E} B. Enriquez, {\em PBW and duality theorems for quantum groups and quantum current algebras},
Journal of Lie Theory, Volume 13 (2003) 21--64.


\bibitem[GTL10]{GTL} S. Gautam, V. Toledano Laredo, {\em Yangians and quantum loop algebras}. Selecta Mathematica 19 (2013) no.
2, 271–336. \arxiv{1012.3687}

\bibitem[GTL16]{GTL16}  S. Gautam, V. Toledano Laredo, {\em Yangians, quantum loop algebras, and abelian difference equations}.
J. Amer. Math. Soc. 29 (2016), 775--824. 


\bibitem[G07]{G} N.~Guay, {\em Affine Yangians and deformed double current algebras in type A}, Adv. Math. \textbf{211} (2007), no.\textbf{2}, 436--484. \MR{2323534}

\bibitem[GNW17]{GNW} N.~Guay, H.~Nakajima, and C.~Wendlandt, {\em Coproduct of Yangians of affine Kac-Moody algebras}, preprint, (2017).\arxiv{1701.05288}.


\bibitem[GRW18]{GW} N.~Guay, V. Regelskis, C.~Wendlandt, {\em Vertex Representations for Yangians of Kac-Moody algebras}, preprint, (2018). \arxiv{1804.04081}





\bibitem[K85]{K}V. ~Kac,
{\em Infinite dimensional Lie algebra}. Cambridge University Press, Cambridge, 1985.
\MR{0823672}

\bibitem[Ne15]{Neg} A.~Negut, {\em Quantum Toroidal and Shuffle Algebras, R-matrices and a Conjecture of Kuznetsov}, \arxiv{1302.6202}.



\bibitem[YZ14]{YZ1} Y.~Yang and G.~Zhao, {\em The cohomological Hall algebra of a preprojective algebra}. Proc. Lond. Math. Soc., to appear. 
\arxiv{1407.7994}

\bibitem[YZ16]{YZ2} Y. Yang and G. Zhao, {\em Cohomological Hall algebras and affine quantum groups},  Selecta Math. Vol. 24, Issue 2 (2018), pp. 1093--1119. \arxiv{1604.01865}

\bibitem[YZ18]{YZ3} Y. Yang and G. Zhao, {\em How to sheafify an elliptic quantum group}, to appear in the MATRIX Annals 2017. \arxiv{1803.06627}.


\end{thebibliography}
\end{document}